\documentclass[11pt]{article}
\vspace{7cm} \textwidth 14cm \textheight 21.6cm
\usepackage{amsmath}
\usepackage{amsfonts}
\usepackage{url}
\usepackage[english]{babel}
\usepackage[dvipsnames]{xcolor}
\usepackage{hyperref}
\usepackage[pdftex]{graphicx}
\usepackage[utf8]{inputenc}
\usepackage{changes}
\usepackage{latexsym,amsthm,amscd}

\newtheorem{prop}{Proposition}

\newtheorem{rem}{Remark}

\newcounter{alphthm}
\theoremstyle{remark}
\usepackage{amssymb, graphicx, tcolorbox, wasysym,  pifont}
\numberwithin{equation}{section}

\newcommand{\RR}{\mathbb R}
\newcommand{\too}{\longrightarrow}
\newtheorem*{theorem*}{Theorem} 
\newtheorem{theorem}{Theorem}[section]
\newtheorem{proposition}[theorem]{Proposition}
\newtheorem{corollary}[theorem]{Corollary}
\newtheorem{lemma}[theorem]{Lemma}

\DeclareMathOperator{\Hessian}{Hess}

\theoremstyle{definition}

\newtheorem{example}[theorem]{Example}

\theoremstyle{remark}

\title{Mathematical modeling the wildfire propagation in a Randers space}

\author{ Hengameh R. Dehkordi \thanks{ORCID ID: 0000-0002-1738-3373, \hfill\break\indent
		Email:hengamehraeesi@gmail.com, hengameh.r@ufabc.edu.br, \hfill\break\indent
		Phone: +55 (11) 4996-8332    	}\\
	  Center of Mathematics, Computation and Cognition, \\
	Federal University of ABC, \\
	 Santo André, Brazil

}

\date{\today}
\date{}

\begin{document}
\maketitle
%
%
%



	\begin{abstract}
	The devastating effects of wildfires on the wildlife and their impact on human lives and properties are undeniable. This shows the importance of studying the  spread of wildfire, predicting its behavior and presenting more reliable models for its propagation.	Here, by using the validity of the Huygens’ envelope principle for wavefronts in Randers spaces, we present some models for the propagation of  wildfire in an  $n$-dimensional smooth manifold under the presence of wind. In the models, trajectories of fire particles are tracked and the equations that give the wavefront  at each time are provided. Furthermore, we determine the paths and points of great importance in the process of wildfire management, called strategic paths and points. Finally, we consider two examples of spreading the wildfire in some agricultural land or woodland, for the sake of illustration.
\end{abstract}



\section{Introduction}

Every year, wildfires cause significant damage to the wildlife, jungles, grasslands,  agricultural lands, and natural resources; and threaten infrastructures, properties, and human lives \cite{dore2010carbon}.               Sometimes, the renovation of destroyed regions and the recovery of damaged faunas are impossible. Indeed, wildfires have strongly negative ecological effects  and every year engulf millions hectares of rainforests \cite{flannigan2006forest}. Global warming  and carbon dioxide released into the atmosphere due to wildfires are other issues that can not be ignored \cite{wiedinmyer2007estimates}.

Providing a more accurate and reliable model for spreading the wildfire in time plays an important role in the wildfire management strategies and without any doubt reduces both the financial  and life losses due to wildfire. Finsler geometry, which is a classical branch of differential geometry started by P. Finsler in 1918, is a strong tool to model some real phenomena in anisotropic or inhomogeneous media \cite{ giambo2009genericity, kopacz2017application, yajima2015finsler}. On the propagation of waves and tracking the wave rays, several authors have already applied the Finsler metric \cite{gibbons2011geometry, markvorsen2016finsler}. As a particular case,  the propagation of wildfire waves in dimension $2$ was studied in \cite{markvorsen2016finsler}, in which the author showed that, for a wildfire spreading, the wavefronts and  wave rays are respectively the geometric spheres and geodesics of the corresponding Finsler metric. Very recently, the validity of Huygens' principle was  verified for wavefronts in Finsler spaces of any  dimension $n$ \cite{dehkordi2019huygens}.

The Huygens' principle is frequently used for modeling the growth of wildfire in dimension~$2$. A literature review shows that several authors used some fixed elliptical template fields, such as double ellipse, lemniskata, oval shape, and tear shape, as spherical wavefronts in the process of applying the Huygens' principle \cite{anderson1982modelling, cunbin2014analysis, glasa2009mathematical,   markvorsen2016finsler, richards1995general}. However, the spherical wavefronts used in this process have some deviations from the spherical wavefront in reality \cite{markvorsen2016finsler}. Consequently, sometimes, the model presented based on these spherical wavefronts is neither accurate enough nor reliable. The reason why these template fields are not sometimes so suitable is that the curvature of space is assumed to be zero while, in reality, it is not zero. Therefore, the more the curvature differs from zero, the less the model  approximates the propagation \cite{hilton2017curvature}. By the way, in the models presented by using the Finsler metric, the curvature of space is taken into account. In fact,  the Finsler metric, by taking the curvature  into account,  provides geometric spheres with negligible deviation from the spherical wavefronts created by the wildfire. 


The use of simulators, such as Phoenix, IGNITE, Bushfire, FireMaster, FARSITE  and Prometheus is another technique widely used by the researchers to predict the behavior of the fire \cite{johnston2005overview, papadopoulos2011comparative, tymstra2010development}. To the best of our knowledge,  FARSITE and Prometheus are considered to be the best ones among these simulators \cite{sullivan2009wildland}. The challenge we confront dealing with simulators is that they produce errors during the process. Therefore, regarding  simulators, new problems appear which are reducing the errors and time of computations \cite{arca2007evaluation}.

In \cite{dehkordi2019huygens}, it was shown  that for a given propagation of waves, if one  finds a Finsler metric $F$ in such a way that one of the preimages of Finsler distance function coincides with the wavefront at some time $t$, then the Finsler distance function provides a model for that propagation. In fact, the Finsler geodesics and  distance function are applied to modeling the propagation of  waves. In our work, by using the results of \cite{dehkordi2019huygens}, we suggest two different strategies to predict the progress of fire waves in the case of Randers metrics. In other words, in order to provide the model of a given propagation, one may use the Randers geometric spheres, and then apply the Huygens' principle, or solve the system of Randers geodesic equations, which are indeed paths of fire particles. By the way, by the spread of wildfire in a Randers space, it means a fire spreading across a smooth manifold $M$ under the presence of some wind $W$. The wind here is a smooth vector field $W$ such that $|W|< 1$.

In this work, through three main results, that is Theorems \ref{main.th.1}, \ref{main.th.2} and \ref{mainthm3}, we provide the equations of wave rays  and wavefronts at each time $\tau$. Moreover, the equations of strategic paths and points are presented. 
Let us explain that by\textit{ strategic paths} we mean the paths along which the fire engulfs more regions or it reaches to some special zone that has some priority (protecting fauna, houses, etc.) and should be protected from the fire. In other words, firefighters and equipment should be located along the strategic paths in order to control the fire or prevent it from progressing toward some special direction or area. In fact, having strategic paths increases the chance of success in wildfire management. After finding the model of propagation and strategic paths, by using some information on the behavior of wildfires, one determines some points where the firefighters and equipment should be located to attack the fire. Such points,  that are along strategic paths, are called \textit{strategic points}. It should be pointed out that time is so important when it comes to responding to an emergency incident and also our sources, forces, and equipment are limited against natural phenomena such as wildfires. Therefore, finding the strategic points, which leads to saving time and expense, is vital to wildfire management, especially in wildfires of  big scale and magnitude that are happening every year around the world, such as the United States, Australia, and Brazil.  Hence, it is fair to claim that the study of   strategic points and paths is of great importance in the firefighting process and really demands more attention. Whereas, to the best of our knowledge, there is no study related to such paths or points. 

\subsection{Hypotheses, methodology, and Outline of the paper}

Throughout this work, it is assumed that a wildfire is sweeping across some space $M$ which is a smooth manifold of dimension $n$ and some fuel has been distributed  uniformly and smoothly through $M$  Also the temperature and moisture are constant at all points of $M$. A mild wind $W$, that is a smooth vector field, is blowing in $M$. Although the wind might or not be space-dependent,   it must be time-independent at intervals of time. Also, it is assumed that the fire is stopped before it creates singularities or cut loci. To be closer to  reality, the focus of our work is on the dimension $3$ (for instance $M$ can be any open subset of $\mathbb{R}^3$), however our results are valid for any dimension $n$.


On the methodology, we show that the model of  propagation is merely determined by considering some translated ellipsoids. This way, it is enough to determine the equation of such ellipsoids from the experimental or laboratory data. Afterwards, equations of wavefronts and wave rays associated with the propagation are determined, and finally the model is presented. Following this strategy, we investigate the spread of  wildfire across some space under the presence of  wind which is a constant, Killing or smooth vector field.

The rest of this paper is organized as follows. In Section 2, we recall some preliminaries on Randers spaces and Hyugens' principle. In Section 3, we proceed with the main results and certain discussions on the spread of fire waves. In Section 4, we provide some examples in which certain wildfire is spreading in some agricultural land, woodland or forest under the influence of wind. In Section 5, concluding remarks and some ideas for future works are provided.


\section{preliminaries}

Here, for the sake of self-contained paper, we first provide a brief review of the Randers geometry, and then recall the Huygens' envelope principle.

Let  $M$ be a smooth manifold, $p=(x_1,...,x_n)\in M$ a point and $T_pM$ the  space tangent to $M$ at $p$. Assume that $V=(v_1,...,v_n)\in T_pM$ is a vector according to the canonical basis $\{\frac{\partial}{\partial x_i}\}_{i=1}^n$ for $T_pM$ and $TM$ the tangent bundle, i.e. the collection of all  vectors tangent to $M$ or in other words $$TM=\underset{p\in M}{\dot{\cup}}\{(p,V): V\in T_pM\}.$$
 A \textit{Riemannian metric} on  $M$ is a smooth function $h$ such that to each point $p\in M$, $h$ assigns a positive-definite inner product $h_p:T_pM\times T_pM\to\mathbb{R}$. The smoothness condition on $h$ refers to the fact that the function $p\in M\to h_p(\frac{\partial}{\partial x_i},\frac{\partial}{\partial x_j})\in\mathbb{R}$ must be smooth.  Now, consider a Riemannian metric $\mathfrak{r}$ on $M$, $\mathfrak{r}:TM\times TM\to[0,\infty)$, and a $1$-form $\beta:TM\to[0,\infty)$ such that $\mathfrak{r}(\beta^*,\beta^*)<1$, where $\beta^*$ stands for the dual vector of $\beta$. Considering $\alpha(.)=\sqrt{\mathfrak{r}(.,.)}$, then $F=\alpha+\beta$ is called the \textit{Randers metric}. It is well known that given any Randers metric $F$ on $M$, there is a Riemannian metric $h$ on $M$ and some vector field $W$,  which
 satisfies $h(W, W) < 1$, related to this Randers metric $F$. We call the pair $(W;h)$ the Zermelo data associated with $F$. Due to Zermelo, the problem which associates the Randers metric $F$ with the Riemannian metric $h$ is called the\textit{ Zermelo's navigation problem}. For a literature review on the Zermelo's problem of navigation see \cite{muresan2014zermelo}.  The Randers metric $F$ and the associated Zermelo data are satisfied in the following equation:
\begin{equation}
	\label{Randers1}
	F(V)= \alpha(V) + \beta(V)= \frac{\sqrt{h^2(W,V)+\lambda h(V,V)}} {\lambda}- \frac{h(W,V)}{\lambda},   
\end{equation} 
where $\lambda =1-h(W,W) $ \cite{bao2004zermelo}. As a result of Eq.~(\ref{Randers1}), one has  
\begin{equation}\label{rel}
	F(V)=1\ \text{if and only if}\ h(V-W,V-W)=1
\end{equation}  (see the details in section $1.1.2$ of \cite{bao2004zermelo}). 

Given a Randers space $(M,F)$ and some point $p\in M$, the \textit{indicatrix of radius} $\tau$ at $p$ is the set $$\mathcal{I}_{_F}^\tau = \{V\in T_pM : F(V) = \tau\}. $$ Indeed, the indicatrix  $\mathcal{I}_{_F}^\tau$ is the collection of all end points of   vectors tangent to $M$ at $p$ and length $\tau$, and therefore it is a hypersurface in $T_pM$. We show the indicatrix of radius $1$ with  $\mathcal{I}_{_F}$. By the way, given the equation of Randers indicatrix, there are some techniques to determine the equation of its Randers metric (for instance, see \cite{bao2012introduction}, p. 13).  Given any Riemannian manifold $(M,h)$, one defines the Riemannian indicatrix of radius $\tau$, $\mathcal{I}_h^\tau$ similar to the Randers indicatrix $\mathcal{I}_{_F}^\tau$.

Given some Randers space $(M,F)$ and a piece-wise smooth curve ${\gamma}:[0,1]\longrightarrow M$, the length of ${\gamma}$ is defined as $L[{\gamma}]:=\int_{0}^{1}F({\gamma^\prime}(t))dt$. Similar to the Riemannian space, the distance from a point $p\in M$ to another point $q\in M$ in the Randers space $(M,F)$ is  
\begin{equation}\label{dis}
	d(p,q):=\inf_{{\gamma}}\int_{0}^{1}F({\gamma^\prime}(t))dt,
\end{equation}
where the infimum is taken over all piece-wise smooth curves ${\gamma}:[0,1] \longrightarrow M$ joining $p$ to $q$. A smooth curve in a Randers manifold is called a \textit{geodesic} (shortly $F$-geodesic) if it is locally the shortest time path connecting any two nearby points on this curve.
Given a compact subset $A\subset M$, we define the \textit{Randers distance function} $\rho:M\too\RR$ with $\rho(p)=d(A,p)$. One proves that $\rho$ is locally Lipschitz continuous  and therefore it is  differentiable almost everywhere \cite{shen2001lectures}.

Given some smooth vector field $W$ on $M$, the flow of $W$  is the smooth map $\varphi:(-\epsilon,\epsilon)\times M \longrightarrow M$ such that for all $p\in M$, $\varphi^p(t):=\varphi(t,p)$ is an integral curve of $W$, that is $\frac{d\varphi^p}{dt}(t)|_{t=0}=W(p)$ \cite{lee2018introduction}. 
A vector field $W$ on a Riemannian space is called \textit{Killing} if and only if its flow is an isometry of $(M,h)$. One can also say that $W$ is \textit{Killing} if and only if $\mathcal{L}_Wh=0$, where $\mathcal{L}$ is the Lie derivative.
\begin{lemma}\label{col}\cite{robles2005geodesics}
	Assume that $(M,h)$ is a Riemannian manifold. Given a unitary Riemannian geodesic ($h$-geodesic) ${\gamma_{_h}} : (-\epsilon,\epsilon) \longrightarrow M$ and a \textit{Killing} vector field $W$, the
	unitary $F$-geodesics are ${\gamma_{_F}}(t) = \varphi(t,{\gamma_{_h}}(t))$, where $\varphi : (-\epsilon,\epsilon) \times U \longrightarrow M$ is the flow of $W$ through ${\gamma_{_h}}(t)$.
\end{lemma}

Given some Randers space $(M,F)$ and a submanifold $A\subset M$, we say that a vector $V$ is \textit{ $F$-orthogonal to $A$}, and write $V\underset{F}{\perp}A$, if for every vector $U$ tangent to $A$ one has $g_V(V,U)=0$, where 	\begin{equation*}\label{fun.form}
	g_{V}(V,U):= \frac{1}{2}\bigg( \frac{\partial^2}{\partial t \partial s}F^2(V+tV +sU)\bigg)_{s=t=0}.
\end{equation*}	 Similarly, given the Riemannian manifold $(M,h)$, a vector $V$ is $h$-orthogonal to $A$, i.e. $V\underset{h}{\perp}A$, if for every vector $U$ tangent to $A$ one has $h(V,U)=0$. The following corollary states the relation between $F$-orthogonality and $h$-orthogonality. 

\begin{corollary}\cite{dehkordi2018finsler}\label{orto}
	Given a Randers manifold $(M,F)$, assume that $(h;W)$ is the Zermelo data associated with it. Then, for any two non-zero vectors $U$ and $V$ tangent to $M$,  $U\underset{F}{\perp}V$ if and only if $h(U,\frac{V}{F(V)}-W)=0$.
\end{corollary}	

\subsection{Huygens' principle}

Assume that $P$ is a source which emits waves. Given any time $t>0$, we consider the collection of all points to which the wave reaches at time $t$. This collection is called the \textit{wavefront} at time $t$ \cite{arnol2013mathematical}. In the case that $P$ is a single point, the wavefront is called the \textit{spherical wavefront}. Given a wavefront $B$, assume that each point on $B$ acts as a new source that emits spherical wavefronts. At any time later, a surface tangent to all of the spherical wavefronts is called the \textit{envelope} of $B$. By the \textit{wave ray} it means the shortest time path that connects any point of $B$  to the wavefront at any time later.  We recall the \textit{Huygens' Theorem} as follows:

\begin{theorem}\cite{arnol2013mathematical}
	Let $\phi_p(t)$ be the wavefront of the point $p$ after time $t$. For every point $q$ of this wavefront, consider the wavefront after time $s$, i.e. $\phi_q(s)$. Then, the wavefront of point $p$ after time $s+t$, $\phi_p(s+t)$, will be the envelope of wavefronts $\phi_q(s)$, for $q\in\phi_p(t)$.
\end{theorem}

We recall the following result from \cite{dehkordi2019huygens} which is used throughout this work several times. By the way, since a Randers metric is a particular case of a Finsler metric, any result on Finsler is valid for Randers, as well. It is assumed that, in a Finsler space $(M,F)$, a wildfire is spreading and sweeping some area $U\subset M$ in the interval of time from $t = s>0$ to $t = r$. It
is also assumed that $U$ is a smooth manifold and $d$ is the Finsler distance function.
\begin{proposition}\label{propag.fire} \cite{dehkordi2019huygens}
	Let $\rho:M\too \RR$ with $\rho(p)=d(A,p)$ where $A$ is a compact subset of $M$ and $\rho(U)=[s,r]$, where $0<s<r$. Suppose that $\rho^{-1}(s)$ is the wavefront at time $0$ and there is no cut loci in $\rho^{-1}(s,r)$. Then,
	for each $t\in [s,r]$, $\rho^{-1}(t)$ is the wavefront at time $t-s$ and 
	the Huygens' principle is satisfied by the wavefronts $$\{\rho^{-1}(t)\}_{t\in [s,r]}.$$
	Furthermore, the track of each fire particle is a geodesic of $F$ and also it is orthogonal to each wavefront $\rho^{-1}(t)$ at time $t-s$.
\end{proposition}
This theorem says that once one finds the Finsler metric associated with a wildfire spreading in a smooth manifold $M$, by using the distance function $\rho$, the model of propagation is provided. As a result from Proposition~\ref{propag.fire}, one has the following corollary.
\begin{corollary}
	Let $\rho:M\too \RR$ with $\rho(p)=d(A,p)$ where $A$ is a compact subset of $M$ and $\rho(U)=[0,r]$. Suppose that there is no cut loci in $\rho^{-1}(0,r)$ and $\rho^{-1}(0)$ is the wavefront at time $0$. Then,
	for each $t\in [0,r]$, $\rho^{-1}(t)$ is the wavefront at time $t$ and 
	the Huygens' principle is satisfied by the wavefronts $$\{\rho^{-1}(t)\}_{t\in [0,r]}.$$
	Furthermore, the track of each fire particle is a geodesic of $F$ and also it is orthogonal to each wavefront $\rho^{-1}(t)$ at time $t$.
\end{corollary}

\begin{proof}
	In fact, since $\rho$ is differentiable almost everywhere It is enough to take the limit of the function $\rho$ to extend the results of Proposition~\ref{propag.fire} to the whole interval $[0,r]$.
\end{proof}

\section{Modeling the propagation by using the Randers geometry}


Throughout this section, we consider the hypothesis stated in the Introduction and establish some results. We consider three different states for the wind:  constant, \textit{Killing}, and  smooth vector filed.


Here, it is shown that for modeling a propagation one first needs to find the equation of some ellipsoid - as several authors have used elliptic fields in the case of dimension $2$ \cite{anderson1982modelling}. From this ellipsoid, that is actually our indicatrix, we calculate the Riemannian metric, then the wave rays, and finally we provide the model.	
Before presenting the main results, we state and prove Lemma~\ref{indic} which says  that the Randers indicatrix of radius $\tau$ is translation of the Riemannian indicatrix of the same radius   by vector $\tau W$.
\begin{lemma}\label{indic}
	Given a Randers space $(M,F)$, let $(h;W)$ be the Zermelo data associated with $F$ and $\mathcal{I}_{_F}^\tau$ and $\mathcal{I}_{_h}^\tau$ the Randers and Riemannian indicatrices of radius $\tau$, respectively. Then,  $\mathcal{I}_{_F}^\tau=\mathcal{I}_{_h}^\tau+\tau W.$
\end{lemma}
\begin{proof}
	One has	$V \in \mathcal{I}_{_h}^\tau$ if and only if $|V|_{_h}=\tau$ if and only if $|\frac{V}{\tau}|_{_h}=1$ if and only if $\frac{V}{\tau}\in \mathcal{I}_{_h}$, where $|\,.\,|_{_h}=\sqrt{h(\,.,.\,)}$.  In other words, $\mathcal{I}_{_h}^\tau=\tau \mathcal{I}_{_h}$.  Similarly, one shows   $\mathcal{I}_{_F}^\tau=\tau\mathcal{I}_{_F}$.		Furthermore, according to the relation \ref{rel}, $F(V)=1$ if and only if $|V-W|_{_h}=1$, or in other words $V\in \mathcal{I}_{_F}$ if and only if $V-W\in \mathcal{I}_{_h}$ which leads to $\mathcal{I}_{_F}=\mathcal{I}_{_h}+W$. Finally, from these relations we have
	\begin{align*}
		\mathcal{I}_{_F}^\tau&=\tau\mathcal{I}_{_F}=\tau(\mathcal{I}_{_h}+W)=\tau\mathcal{I}_{_h}+\tau W\\
		&=\mathcal{I}_{_h}^\tau+\tau W,
	\end{align*}
	which is the desired relation.
\end{proof}

From this lemma, given a Randers indicatrix $\mathcal{I}_F$, one can find the Riemannian indicatrix $\mathcal{I}_h$, and vice versa. Consequently, assuming that $Q_h$ is  the quadratic equation of $\mathcal{I}_h$, the Riemannian metric  is  $h=\frac12 \Hessian Q_h$ and, by using Eq.~(\ref{Randers1}), one finds the Randers metric $F$.

\subsection{Constant wind}
In this section, for a constant wind $W$, we first find the wavefronts and wave rays of the propagation, and then the strategic paths and points. 

\subsubsection{Wavefronts and Wave rays}
\begin{theorem}\label{main.th.1}
	Assume that a wildfire is spreading  in some space $ M$ while the wind $W=(0,W_2,W_3)$ is blowing across $M$ and $A$ is the wavefront at time $0$. Then:
	\begin{itemize}
		\item [(i)] Given any point $p$ in $M$, the spherical wavefront of  radius $\tau$ and center $p$ is 
		\begin{align}
			Q_{_h}(\frac{u}{\tau},\frac{v}{\tau},\frac{w}{\tau})+\tau W+p,
		\end{align} 
		where $Q_h(u,v,w)$ is given by the Eq.~(\ref{el}).

		\item [(ii)] Given any point $p$ in $A$, the wave ray emanating from $p$ is the straight line given by ${\gamma_{_F}}(t)=p+tV$, where $V$ is a vector such that $|V-W|_{_h}=1$ and $V-W\underset{h}{\perp}A$, where $h=\verb*|Hess|\, Q$. 
		\item [(iii)] The wavefront at time $\tau$ is the set \[\{p+\tau V\ : \ p\in A, |V-W|_{_h}=1, V-W \underset{h}{\perp} A \}.\]

	\end{itemize}
\end{theorem}

\begin{proof} 
	
	\textbf{$(i):$}	First of all, observe that we have the same conditions at all  points of $M$. Hence, the spherical wavefronts of the radius $\tau$ and centers at any points of $M$ are the same geometric objects, up to a translation. Consequently, the metric $F$ associated with the spherical spheres does not depend on the point, however it depends on the direction. In other words, $F$ is a Minkowski-Randers type. As a  consequence, the Randers geodesics are straight lines and  the space is of zero curvature. Therefore, the spherical wavefront of radius $\tau$ and center  $p$ coincides with  $\mathcal{I}^\tau_{_F}$, that is the indicatrix of radius $\tau$ and center $0_p\in T_pM$. Let us find the equation of $ \mathcal{I}^\tau_{_F}$. Assume that $h$ is the metric in  Zermelo data $(h;W)$ whose Zermelo's solution is $F$. By Lemma~\ref{indic}, $\mathcal{I}_h$ is a translation of $\mathcal{I}_F$ by $-W$. Therefore, let's find $\mathcal{I}_h$, in order to find $\mathcal{I}_{_F}$. Assume that $Q_{_h}$ is the quadratic equation of $\mathcal{I}_{_h}$, that is $$Q_{_h}(u,v,w)=[u,v,w]\hbar[u,v,w]^T=1,$$ where $\hbar$ is the matrix of the metric $h$ and $B^T$ is the transpose of the matrix $B$. Observe that $\hbar$ is a symmetric and positive-definite matrix, since it is the matrix of the metric $h$.  Also, the fact that $Q_{_h}$ at each point has the same equation implies that  $\hbar$ has constant components.   We  show through the following lemma that $Q_{_h}$ is a rotated ellipsoid.
	
	\begin{lemma}\label{ind-eli}
		Let  $Q_{_h}(u,v,w)$ be a quadratic equation such that $$Q_{_h}(u,v,w)=[u,v,w]\hbar[u,v,w]^t=1,$$ where $\hbar$ is a symmetric and positive-definite matrix with constant components. Then, $Q_{_h}$ is the equation of a rotated ellipsoid.
	\end{lemma}
	\begin{proof}
		Since $\hbar$ is a symmetric and positive-definite matrix,  by Spectral theorem, there exists some orthogonal matrix $P$ (i.e. $PP^T=P^TP=I$) such that $\hbar=P^T\mathcal{D}P$, where $\mathcal{D}=diag(\lambda_1,\lambda_2,\lambda_3)$ is a diagonal matrix. Therefore, one can write
		
		\begin{equation*}
			\begin{array}{ll}
				1=\begin{bmatrix} u \\ v \\ w  \end{bmatrix}^TP^T\mathcal{D}P\begin{bmatrix} u \\ v \\ w \end{bmatrix} &= (P\begin{bmatrix} u \\ v \\ w \\ \end{bmatrix})^T\mathcal{D}P\begin{bmatrix} u \\ v \\ w \\ \end{bmatrix}\\
				&\\
				&= \begin{bmatrix} u' \\ v' \\ w' \\ \end{bmatrix}^T\mathcal{D}\begin{bmatrix} u' \\ v' \\ w' \\ \end{bmatrix}
				=\lambda_1{u^\prime}^2+ \lambda_2{v^\prime}^2+\lambda_3 w{^\prime}^2,
			\end{array}
		\end{equation*}
		where we put $P\begin{bmatrix} u \\ v \\ w \\ \end{bmatrix}=\begin{bmatrix} u' \\ v' \\ w' \\ \end{bmatrix}$ and the last equality is because $\mathcal{D}$ is a diagonal matrix.   Furthermore, since $\hbar$ is positive definite, the elements on the principle diagonal of $\mathcal{D}$ are positive real numbers and we can write
		
		\begin{equation}\label{ell}
			1= \lambda_1{u^\prime}^2+ \lambda_2{v^\prime}^2+\lambda_3 w{^\prime}^2=(\frac{u^\prime}{a})^2+ (\frac{v^\prime}{b})^2+(\frac{w^\prime}{c})^2,
		\end{equation} where $a=\frac{1}{\sqrt{\lambda_1}}$, $b=\frac{1}{\sqrt{\lambda_2}}$, $c=\frac{1}{\sqrt{\lambda_3}}$. Eq.~(\ref{ell}) is the equation of an ellipsoid in the vector space $T_pM$ with  basis $\mathbb{B}=\{P\frac{\partial}{\partial x},P\frac{\partial}{\partial y},P\frac{\partial}{\partial z}\}$, which is the rotation of the canonical basis. Therefore, it is deduced that $Q_h$ is some rotated ellipsoid in the system with basis of $\{\frac{\partial}{\partial x},\frac{\partial}{\partial y},\frac{\partial}{\partial z}\}$. 
	\end{proof}

	From Lemma~\ref{ind-eli}, the quadratic equation of spherical wavefronts  with respect to the Riemannian metric $h$,  $Q_h(u,v,w)$, is some rotated ellipsoid given by Eq.~(\ref{ell}). 		Let us figure out what is the angle of  rotation. Observe that, according to the following  facts: the heat goes toward above, i.e. $z$-axis, the fuel has been distributed uniformly through $M$, and $W$ is in the $yz$-plane; it is deduced that $Q_h$ is rotated around $x$-axis. That is the matrix of rotation is 
	
	\begin{align} 	\label{prot}	
		P=R_x(\alpha)=
		\left(\begin{matrix}
			1 &           0 &            0\\
			0 & \cos \alpha & -\sin \alpha\\
			0 & \sin \alpha & \cos \alpha
		\end{matrix}\right).\end{align}
	Here, we used a right-handed coordinate system and a right-handed rotation through an angle $\alpha$ around $x$-axis.   To summarize we have: 
	\begin{align}\label{el}
		Q_{_h}(u,v,w)=(\frac{u}{a})^2+ (\frac{v\cos\alpha-w\sin\alpha}{b})^2+ (\frac{v\sin\alpha+w\cos\alpha}{c})^2=1,
	\end{align}
	were $a,b$, $c$ and $\alpha$ are constant real numbers and will be determined from the experimental data.			Finally, the Riemannian indicatrix is  
	\begin{equation*}
		\mathcal{I}_{_h}=\{(u,v,w)\ : Q_{_h}(u,v,w)=1\}
	\end{equation*} 
	and therefore, by Lemma~\ref{indic}, one  obtains 
	\begin{equation}\label{ind.f}
		\mathcal{I}_{_F}^\tau=\{(u,v,w) \in \mathbb{R}^3\ : Q_{_h}(\frac{u}{\tau},\frac{v}{\tau},\frac{w}{\tau})=1\}+\tau W.
	\end{equation}
	Eq.~(\ref{ind.f}), that is the equation of  spherical wavefront at time $\tau$ and center $p$, is  translation of the rotated ellipsoid given by Eq.~(\ref{el}), in which the constant numbers are determined from the experimental data, and so the proof of item $(i)$ is complete.

	\vspace{3mm}  \textbf{$(ii)$:} Given some point $p\in A$, by Proposition~\ref{propag.fire} and  since geodesics of Minkowski spaces are straight lines,  the wave rays emanating from the point $p$  are unit speed straight lines ${\gamma_{_F}}(t)=p+tV$, provided that vectors $V$ are $F$-orthogonal to $A$. Therefore, by Corollary \ref{orto} and relation \ref{rel}, we must have $V-W\underset{h}{\perp}A$  and $|V-W|_{_h}=1$. 
	
	\vspace{3mm} \textbf{$(iii)$:} By Proposition~\ref{propag.fire}, the wavefront at time $\tau$ is $\rho^{-1}(\tau)$, where $\rho:M\to\mathbb{R}$, $\rho(p)=d(A,p)$.  Given $q\in \rho^{-1}(\tau)$, assume that ${\gamma_{_F}}$ is a unit speed geodesic (wave ray) that minimizes the distance from $A$ to $q$. That is it emanates from some point $p\in A$ and reaches to $q$ at time $\tau$, or in other words, $$\tau=d(A,q)=d(p,q)=d({\gamma_{_F}}(0),{\gamma_{_F}}(\tau)).$$ Note that, by item $(ii)$, ${\gamma_{_F}}$ is the straight line ${\gamma_{_F}}(t)=p+tV$ with $|V-W|_{_h}=1$ and $V-W\underset{h}{\perp}A$. It concludes that \[\rho^{-1}(\tau)=\{q={\gamma_{_F}}(\tau)=p+\tau V\ : \ p\in A, |V-W|_{_h}=1, V-W \underset{h}{\perp} A \}.\]

\end{proof}


%
\begin{corollary}\label{huy.cor}
Assume that a wildfire is spreading  in some space $ M$ while the wind $W=(0,W_2,W_3)$ is blowing across $M$ and $\mathcal{VF}_{\tau_1}$ is the wavefront at time $\tau_1\geq 0$. 	Then the wavefront at time $\tau_2$ is the envelope of ellipsoids $Q_{_h}(\frac{u}{\tau},\frac{v}{\tau},\frac{w}{\tau})+\tau W+p,$ where  $\tau=\tau_2-\tau_1$, $p\in \mathcal{VF}_{\tau_1}$, and $Q_{_h}(u,v,w)$ is given by Eq.~\ref{el}.
\end{corollary}
\begin{proof} 
	We consider the Randers distance function $\rho:M\to\mathbb{R},$ $\rho(p)=d(\mathcal{VF}_{0},p)$, where $\mathcal{VF}_0$ is the wavefront at time $0$.
Therefore, the proof is a direct consequence of   Proposition~\ref{propag.fire}  and item $(i)$ of Theorem~\ref{main.th.1}.
\end{proof}

It is notable that there may not be any relations between the angle $\alpha$ and  angles that the wind makes with the axes. We  can only say that the stronger the wind, the closer the angle $\alpha$ to the angle that the wind makes with the $z$-axis. In fact,  since the heat goes towards above, the indicatrix is an ellipsoid with the major axis along the  $z$-axis, before the influence of the wind. Then, the wind makes it  rotate and translate toward the direction of $W$. Consequently, for the stronger wind we have the smaller deviation between the angle of rotation and the direction of $W$. 

\subsubsection{Strategic paths and points in the case of  constant wind}

Here, we determine the equations of strategic paths for two different situations. The situations are as follows. 
\begin{itemize}
	\item [1-] All the points of $M$ have the same priority from the fire fighting point of view. 
		
	\hspace{-10mm}Therefore, the strategic path is the path along which the fire engulfs more area. 
	\item[2-] Some special area or point has higher priority and  the objective is protecting it against the wildfire.
\end{itemize}  Hence, for a given point, the strategic path is the path through which the fire particles reach to this point at some time $\tau$. Or, for a given area $B$, the strategic path is the path through which the fire meets $\partial B$ for the first time at time $\tau$, where $\partial B$ is the frontier of $B$. By the way, depending on the case, we may have more than one strategic path. 
\begin{prop}\label{stra.1}
	Assume that a wildfire is spreading  in the space $ M$, the constant wind $W=(0,W_2,W_3)$ is blowing across $M$, and $A$ is the wavefront at time $0$. Then:
	
	\begin{itemize}
		\item [(i)] In the case that all points of $M$ have the same priority, the strategic path is $\gamma_{_F}(t)=p+tV$ provided that $p\in A$ and the Euclidean norm of $V$, i.e. $v_1^2+v_2^2+v_3^2$, is the maximum  among all the vectors $V$ satisfying $ |V-W|_{_h}=1$ and $ V-W \underset{h}{\perp} A $. 
		
		\item [(ii)]			Given any point $q$, the strategic path that reaches to $q$ is $\gamma_{_F}(t)=q +(t-\tau)V$ where $|V-W|_{_h}=1,$  $ V-W \underset{h}{\perp} A$, and $\tau$ is the time of  the wavefront to which $q$ belongs.
		
		\item [(iii)] Given any area $B$, the strategic path which reaches to $B$ is $\gamma_{_F}(t)=q +(t-\tau)V$ such that $ |V-W|_{_h}=1, V-W \underset{h}{\perp} A $, and   $\tau$ is the time of the wavefront that intersects $ B$ for the first time and $q$ is the point of intersection.
	\end{itemize}
\end{prop}

\begin{proof}
	\textbf{$(i)$:} First, since the strategic path is some wave ray of the fire that emanates from $A$, by the item $(ii)$ of Theorem~\ref{main.th.1}, it  must be $\gamma_{_F}(t)=p+tV$ such that $p\in A, |V-W|_{_h}=1,$ and $  V-W \underset{h}{\perp} A $. Next, the fact that the strategic path is the wave ray through which the fire engulfs more region implies that the Euclidean velocity of such a wave ray is the maximum among all these paths $\gamma_{_F}$. In other words, $v_1^2+v_2^2+v_3^2$ is the maximum among all such vectors $V$ satisfying in above conditions.

	\vspace{3mm} \textbf{$(ii)$: } For a given point $q$, assuming that $\tau$ is the time when the wavefront meets $q$ for the first time, there exists a unique wave ray that emanates from some point belonging to $A$ and reaches to $q$ at time $\tau$ (up to a reparameterization). Because the wave rays are  integral curves of the gradient of distance function \cite{dehkordi2019huygens}. Hence, from item $(ii)$ of Theorem~\ref{main.th.1}, the path of this wave ray must be a unit speed straight line that emanates form $A$, $F$-orthogonally, and reaches to the wavefront at time $\tau$. Therefore, it is not difficult to show that $\gamma_{_F}(t)=q +(t-\tau)V$ where $|V-W|_{_h}=1,$ and $ V-W \underset{h}{\perp} A$.
	
	\vspace{3mm}  \textbf{$(iii)$: } First we find the wavefront that meets $ B$ for the first time and then the intersection of this wavefront  and $ B$. Assuming that $q $ is the point of intersection, the rest of the proof is similar to that of item $(ii)$.
	
\end{proof}

\subsection{Wind as Some Smooth Vector Field}
Here, through Theorem~\ref{main.th.2}, we provide the model for the case that the wind is a \textit{ Killing vector field}. The positive aspect of this kind of vector field is that there exists a direct relation between the wave rays of propagation and geodesics of the  Riemannian metric $h$. Therefore, in order to find the wave rays, one just needs to solve the system of $h$-geodesic equations. To see the system of geodesic equations  of a Riemannian metric, see Chapter $6$ of \cite{lee2018introduction}. In Theorem~\ref{mainthm3}, we provide the propagation model for the case that the wind is some smooth vector field, not necessarily \textit{Killing}.
\begin{theorem}\label{main.th.2}
	
	Assume that a wildfire is spreading  in the space $M$, the wind $W$, which is a \textit{Killing} vector filed, is blowing across $M$  and $A$ is the wavefront at time $0$. Then:
	\begin{itemize}
		\item [(i)] Given any point $p$ in $A$, the wave rays emanating from $p$ are ${\gamma_{_F}}(t):=\varphi(t,{\gamma_{_h}}(t))$, where $\varphi$ is the flow of $W$ and $\gamma_{_h}$ is the $h$-geodesic such that ${\gamma_{_h}}(0)=p$, $|{\gamma_{_h}}^\prime(t)|=1$, and $d\varphi_{_p}{\gamma'_{_h}}(0)\underset{h}{\perp}A$. 
		\item[(ii)] The spherical wavefront at time $\tau$ and center of some point $p\in M$ is the set
		
		$\{\varphi(\tau,{\gamma_{_h}}(\tau)): {\gamma}_{_h}\text{ is the unit speed $h$-geodesic that } {\gamma}_{_h}(0)=p\}$.
		\item[(iii)] The wavefront at time $\tau$ is the set $$\{\varphi(\tau,{\gamma_{_h}}(\tau)): \gamma_{_h} \text{ is an $h$-geodesic that } d\varphi_{_p}{\gamma_{_h}}^\prime\underset{h}{\perp} A\,\text{and}\, |{\gamma_{_h}}^\prime|=1\}.$$
	\end{itemize} 	
	
\end{theorem}

\begin{proof}
	Before proceeding with the proofs of items $(i), (ii)$, and $(iii)$, let's see what  the indicatrix $Q_h$ and Riemannian metric $h$ are. Given $p\in M$, one is faced with Theorem~\ref{main.th.1} on $T_pM$ with constant wind $W(p)$ on it. Although, since the direction of $W$ depends on $p$, one cannot assume that $W$ is always in $yz$-plane. That is the indicatrix (as a subset of $T_pM$) is some translation of  an ellipsoid rotated along $x$-, $y$-, and $z$- axes. 	Therefore, the indicatrix  is the translation of  the rotated ellipsoid  
	\begin{align}
		Q_{_h}(p;(u,v,w))=(\frac{u^\prime}{a(p)})^2+ (\frac{v^\prime}{b(p)})^2+ (\frac{w^\prime}{c(p)})^2=1,
		\label{ellipp}	\end{align} by the vector $W(p)$, where $\begin{bmatrix} u' \\ v' \\ w' \\ \end{bmatrix}=P(p)\begin{bmatrix} u \\ v \\ w \\ \end{bmatrix}$,   $P(p)=R_z({\theta(p)})R_y({\beta(p)})R_x({\alpha(p)})$, and 
		\begin{align*} 		
		R_x(\alpha(p))=
		\left(\begin{matrix}
			1 &           0 &            0\\
			0 & \cos \alpha(p) & -\sin \alpha(p)\\
			0 & \sin \alpha(p) & \cos \alpha(p)
		\end{matrix}\right),
		R_y(\beta(p))= 
		\left(\begin{matrix}
			\cos \beta(p) & 0&\sin \beta(p) \\
			0 &            1  & 0\\
			-\sin \beta(p) & 0&\cos \beta(p)  
		\end{matrix}\right),\end{align*}
	\begin{align*}
		R_z(\theta(p))=
		\left(\begin{matrix}
			\cos \theta(p) & -\sin \theta(p) & 0\\
			\sin \theta(p) & \cos \theta(p)  & 0\\
			0 &         0  & 1
		\end{matrix}\right).
	\end{align*} 
Here $a,b,c,\alpha, \beta,$ and $\theta$ are smooth functions and at each point $p$ they are determined from the experimental data. By the way, since $P(p)\in SO(3)$, $\mathbb{B}=\{P\frac{\partial}{\partial x},P\frac{\partial}{\partial y},P\frac{\partial}{\partial z}\}$ is a rotation of the canonical basis $\{\frac{\partial}{\partial x},\frac{\partial}{\partial y},\frac{\partial}{\partial z}\}$ and therefore it is an orthogonal basis for the space. 
	It is not difficult to show that the matrix of Riemannian metric is   $\hbar(p)=(P^T\mathcal{D}P)(p)$, where $\mathcal{D}(p)=diag(\lambda_1(p),\lambda_2(p),\lambda_3(p))$.


	\vspace{3mm}\textbf{ $(i)$:}	First, by Proposition~\ref{propag.fire}, the wave rays are unitary $F$-geodesics that are $F$-orthogonal to $A$. To find unitary $F$-geodesics, by Lemma~\ref{col} it is enough to find  unitary $h$-geodesics.  Therefore, for any $F$-geodesic $\gamma_{_F}(t)$ and its corresponding $h$-geodesic ${\gamma_{_h}}(t)$, from this property of flow that $\varphi(0,p)=p$,   one has ${\gamma_{_F}}(0)=\varphi(0,{\gamma_{_h}}(0))={\gamma_{_h}}(0)$; that means both ${\gamma_{_F}}$ and $\gamma_{_h}$ have the same initial point. Hence, given $p\in A$, to find the wave rays   emanating from $p$, it suffices to find the unitary $h$-geodesics $\gamma_{_h}(t)$  emanating from $p$ provided that $d\varphi_{_p}{\gamma'_{_h}}(0)\underset{h}{\perp}A$. The last relation satisfies the condition that $\gamma_{_F}(t)$ must be $F$-orthogonal to $A$. Because by Corollary \ref{orto}, ${\gamma_{_F}}^\prime(0)\underset{F}{\perp}A$ if and only if  ${\gamma_{_F}}^\prime(0)-W\underset{h}{\perp}A$. Furthermore, from the chain rule we have $${\gamma_{_F}}^\prime(0)=d\varphi(0,{\gamma_{_h}}(0)) +d\varphi(0,{\gamma_{_h}}(0))  {\gamma_{_h}}^\prime(0)=W(p)+d\varphi_{_p}{\gamma_{_h}}^\prime(0).$$ 
	

	\vspace{3mm}	\textbf{ $(ii)$:} By Proposition~\ref{propag.fire},  the spherical wavefront of  radius $\tau$ and center $p$ is the set $\rho^{-1}(\tau)$, where $\rho(.):=d(p,.):M\to\mathbb{R}$. 
	Given $q\in \rho^{-1}(\tau)$, assume that  ${\gamma_{_F}}(t)$ is the equation of a wave ray from $p$ to $q$. Hence, by Proposition~\ref{propag.fire}, ${\gamma_{_F}}(t)$  is a unitary $F$-geodesic that minimizes the distance from $p$ to $q$, or in other words $$d(p,q)=d(\gamma_{_F}(0),\gamma_{_F}(\tau))=\tau.$$ Therefore, $$\rho^{-1}(\tau)=\{{\gamma_{_F}}(\tau) : {\gamma_{_F}}(t) \text{ is an $F$-geodesic that } d(p,{\gamma_{_F}}(\tau))=\tau, F({\gamma_{_F}}^\prime)=1, {\gamma_{_F}}(0)=p \}.$$ Now,  by item $(i)$, ${\gamma_{_F}}(t)=\varphi(t,{\gamma_{_h}}(t))$ where ${\gamma_{_h}}(t)$ is a unit speed $h$-geodesic such that ${\gamma_{_h}}(0)=p$, concluding the proof of item $(ii)$.
	
	
	\vspace{3mm}\textbf{ $(iii)$:} The wavefront at time $\tau$, by Proposition~\ref{propag.fire}, is the set $\rho^{-1}(\tau)$, where $\rho(.)=d(A,.):M\to\mathbb{R}$. In other words, $$\rho^{-1}(\tau)=\{q: \ d(A,q)=\tau\}.$$ Suppose ${\gamma_{_F}}(t)$ is the equation of some unit speed $F$-geodesic that minimizes the distance from $A$ to $q\in \rho^{-1}(\tau)$. That is ${\gamma_{_F}}(0)\in A$, $q={\gamma_{_F}}(\tau)$, $\gamma_{_F}'(0)\underset{F}{\perp}A$ and $$d(A,q)=d({\gamma_{_F}}(0),{\gamma_{_F}}(\tau))=\tau.$$ Moreover, by item $(i)$, ${\gamma_{_F}}(t)=\varphi(t,{\gamma_{_h}}(t))$ where ${\gamma_{_h}}(t)$ is a unitary $h$-geodesic such that  $d\varphi_{_p}{\gamma'_{_h}}(0)\underset{h}{\perp}A$.  Therefore, we have  $$\rho^{-1}(\tau)=\{\varphi(\tau,{\gamma_{_h}}(\tau)):\, \gamma_{_h}(t) \text{ is an }  \text{h-geodesic},\, |{\gamma_{_h}}^\prime|_{_h}=1, d\varphi_{_p}{\gamma_{_h}}^\prime\underset{h}{\perp}A\,\},$$ which concludes the proof.
	
\end{proof}

\begin{rem}
	We have a corollary similar to Corollary \ref{huy.cor}. That is to model a wildfire propagation in some space $M$ under the presence of the wind $W$ which is \textit{Killing}, one can find the spherical wavefronts,  centers at different points of some wavefront, by item $(ii)$ of Theorem \ref{main.th.2} and then the envelope.
\end{rem}

\subsubsection{Strategic paths and points in the case of  Killing vector field}
Through Proposition~\ref{stra.2} below, we verify the strategic paths and points in the case  of  wind being a \textit{Killing} vector field.
\begin{prop}\label{stra.2}
	Assume that some wildfire is spreading  in the space $M$ and the wind $W$, that is a \textit{Killing} vector filed, is blowing throughout $M$  and $A$ is the wavefront at time $0$. Then:	
	
	\begin{itemize}
		\item [(i)] In the case that all  points of $M$ have the same priority, given any time $\tau$, the strategic path until time $\tau$ is  $ \varphi(t,{\gamma_{_h}}(t))$, where $\gamma_{_h}$ is the $h$-geodesic such that  $|{\gamma_{_h}}(\tau)-{\gamma_{_h}}(0)|$ is maximum among all  $h$-geodesics that  $|{\gamma_{_h}}^\prime(t)|_{_h}=1$ and $d\varphi_{_p}{\gamma'_{_h}}\underset{h}{\perp}A$.
		
		\item [(ii)]			Given any point $q$ belonging to the wavefront at time $\tau$, the strategic path which reaches to $q$ is  $ \varphi(t,{\gamma_{_h}}(t))$, where   $\gamma_{_h}$ is the  $h$-geodesic such that $|{\gamma_{_h}}^\prime(t)|=1$,  $d\varphi_{_p}{\gamma'_{_h}}\underset{h}{\perp}A$ and $\varphi(\tau,{\gamma_{_h}}(\tau))=q$.
		
		\item [(iii)] Given any area $B$,  the strategic path that reaches to it is $\varphi(t,{\gamma_{_h}}(t))$  where  $\gamma_{_h}$ is the unitary $h$-geodesic such that  $d\varphi_{_p}{\gamma'_{_h}}(0)\underset{h}{\perp}A$ and $\varphi(\tau,{\gamma_{_h}}(\tau))=q$. Here,   $\tau$ is the time of the wavefront that intersects $ B$ for the first time and $q$ is the point of intersection.
	\end{itemize}
	
	%
	%
	
\end{prop}
\begin{proof}
	To prove item $(i)$, each strategic path must be an $F$-geodesic. Therefore, by item $(i)$ of Theorem~\ref{main.th.2}, the strategic path must be  $\varphi(t,{\gamma_{_h}}(t))$, where $\varphi$ is the flow of $W$ and $\gamma_{_h}$ is the unitary $h$-geodesic such that $d\varphi_{_p}{\gamma'_{_h}}(0)\underset{h}{\perp}A$. However, as the wave rays might be some curves, one can not claim that there exists a unique strategic path that remains valid for any time $t$. Indeed, a wave ray might be a strategic path just before some time $\tau$ and after this time one has to choose another wave ray as the strategic path. Since the strategic path is the path through which the fire engulfs more area,  the Euclidean length of the strategic path is the maximum. That is $|{\gamma_{_h}}(\tau)-{\gamma_{_h}}(0)|$ must be maximum among all wave rays joining $A$ to the wavefront at time $\tau$, closing the proof of item $(i)$. 
	
	\vspace{3mm} 	To prove item $(ii)$, similar to item $(i)$, the strategic path is $\varphi(t,{\gamma_{_h}}(t))$, where $\gamma_{_h}$ is the unitary $h$-geodesic such that $d\varphi_{_p}{\gamma'_{_h}}(0)\underset{h}{\perp}A$. However, among all such wave rays, we have to find the one that passes through $q$ at time $\tau$. In other words, the wave ray that satisfies the condition $\varphi(\tau,{\gamma_{_h}}(\tau))=q$ is what we need.
	
	\vspace{3mm} 	The proof of item $(iii)$, is similar to that of item $(iii)$ of Proposition~\ref{stra.1}.
\end{proof}

In the next theorem, we present the model for wildfire propagation when the wind $W$   is a smooth vector field, not necessarily  \textit{Killing}.
\begin{theorem}\label{mainthm3}
	Assume that some wildfire is spreading in a space $M$ across which the wind $W$, that is a smooth vector field, is blowing. Suppose that $A$ is the wavefront at time $0$. Then:
	\begin{itemize}
		\item [(i)] The wave rays are unit speed $F$-geodesics that are $F$-orthogonal to $A$, where $F$ is the Randers metric whose indicatrix at each point $p\in M$ is translation of the rotated ellipsoid, given by Eq.~(\ref{ellipp}), by  $W(p)$.
		\item[(ii)] Given some point $p\in M$, the spherical wavefront of radius $\tau$ and  center $p$ is:
		
		$\{{\gamma_{_F}}(\tau): {\gamma_{_F}}(\tau) \text{ is the unit speed $F$-geodesic that }{\gamma_{_F}}(0)=p\}$.
		\item[(iii)] The wavefront at time $\tau$ is the set: $$\{{\gamma_{_F}}(\tau): {\gamma_{_F}}(\tau) \text{ is the unit speed $F$-geodesic that }{\gamma_{_F}}^\prime(0)\underset{F}{\perp}A\}.$$
		
	\end{itemize}
\end{theorem}
\begin{proof}
	
	To prove item $(i)$, it should be noted that once one has the Randers metric $F$ on $M$, by Proposition~\ref{propag.fire}, the wave rays are unitary $F$-geodesics that are $F$-orthogonal to $A$.  The metric $F$ is given by Eq.~(\ref{Randers1}) which demands finding the Riemannian metric $h$. To find $h$, one has to find the equation of Riemannian/Randers indicatrix, as it was explained after Lemma~\ref{indic}. By following the same argument as that of Theorem~\ref{main.th.2}, it can be shown that at each point $p\in M$ the Randers indicatrix is  translation of the rotated ellipsoid, given by Eq.~(\ref{ellipp}), by the vector $W(p)$. It completes the proof of item $(i)$. 
	
	
	
	
	The proofs of items $(ii)$ and $(iii)$ are similar to those of items $(ii)$ and $(iii)$ of Theorem~\ref{main.th.2}, respectively, in which one does not involve $\gamma_{_h}(t)$ and  $\varphi(.,.)$ in the arguments.
\end{proof}

\subsubsection{Strategic paths and points in the case of  wind being a smooth vector field}
Here, we verify the strategic paths and points in the case that wind is not necessarily a \textit{Killing} vector field.
\begin{corollary}\label{stra.3}
	Assume a wildfire is spreading  in some space $M$, the wind $W$, that is a smooth vector filed, is blowing throughout $M$  and $A$ is the wavefront at time $0$. Then:	
	
	\begin{itemize}
		\item [(i)] If all the points of $M$ have the same priorities, given any time $\tau$, the strategic path till time $\tau$ is   the unitary $F$-geodesic $\gamma_{_F}(t)$ that is $F$-orthogonal to $A$, provided that  $|{\gamma_{_F}}(\tau)-{\gamma_{_F}}(0)|$ is maximum among all such $F$-geodesics. 
		
		\item [(ii)]			Given any point $q$ belonging to the wavefront at time $\tau$, the strategic path which reaches to $q$ is the   unitary  $F$-geodesic $\gamma_{_F}(t)$  that is $F$-orthogonal to $A$ and ${\gamma_{_F}}(\tau)=q$.
		
		\item [(iii)] Given any area $B$, the strategic path that meets $B$ for the first time is the unitary $F$-geodesic $\gamma_{_F}(t)$ which is $F$-orthogonal to $A$ and ${\gamma_{_F}}(\tau)=q$.  Here,   $\tau$ is the time of the wavefront that intersects $ B$ for the first time and $q$ is the point of intersection.
	\end{itemize}
	
	%
	%
	
\end{corollary}
\begin{proof}
	One applies Theorem~\ref{mainthm3} and follows a proof similar to that of Proposition~\ref{stra.2}, in which the flow of $W$ and Riemannian geodesic $h$ do not get involved. 
\end{proof}

\section{Examples}

Here, we provide two examples in which we assume that a wildfire is spreading across some agricultural land or woodland $M\subset \mathbb{R}^3$. In fact we assume that $M$ is some open subset of $\mathbb{R}^3$. It is assumed that the fuel has been distributed smoothly across $M$ and some wind $W$ is blowing throughout $M$.  In Example \ref{exa1}, the wind is constant and in Example \ref{exa2}, it is a \textit{Killing} vector field. In both examples, the fire starts from  ${A}\subset M$, which is a certain point, smooth curve or cylinder. The objective in these examples is finding the wavefronts, wave rays and strategic paths. 		
It is assumed that $M$ contains some flat field of fuel $D$, that is $D$ is a $2$-dimensional subspace of $M$ such that its slope is negligible. In fact, if one walks around in $D$, he would feel some going up and down but negligible. For instance, $D$ might be the bed of agricultural land, a forest, and so forth. 
\begin{example}\label{exa1} 
	Assume that a wildfire is spreading in some agricultural land that $M$ and the wind  $W=(0,1/3,1/6)$ is blowing across $M$. We consider three different cases for $A$, that is the wavefront at time $0$, as follows:
	\begin{itemize}
		\item [Case 1.]  $A$ is some point in $D$.
		\item [Case 2.]  $ A$ is the path of closed curve\begin{equation}\label{case1}
			C(s)=\big(\frac{1}{4}\cos s(\cos s+ 6),\frac{4}{13}\sin s(-\sin s+3),0\big), s\in [0,2\pi].
		\end{equation}
		\item[Case 3.]  $A$ is the image of  surface \begin{equation}\label{case2}
			\hspace{-9mm}S(s_1,s_2)=\big(\frac{1}{4}\cos s(\cos s+ 6),\frac{4}{13}\sin s(-\sin s+3), s_2\big), s_1\in[0,2\pi], \ s_2\in [0,2].
		\end{equation}
	\end{itemize}
	
	Before verifying the cases, let's see what the Riemannian metric and indicatrix are.				 Since the wind is a constant vector field, by Theorem~\ref{main.th.1}, the indicatrix is  translation of the rotated ellipsoid given by Eq.~(\ref{el}), by  $W$. Assume that, from  experimental data, we are given the constant numbers in Eq.~(\ref{el}) as follows:
	\begin{align*}
		a=1/2,\ b=1\ ,c=2, \hspace{2mm} &  \alpha=\pi/6.
	\end{align*}
	Therefore, the equation of the spherical wavefront is 
	%
	%
	%
	Eq.~(\ref{ind.ape}) below:
	\begin{equation}\label{ind.ape}
		64u^2+13v^2+7w^2-(26/3+\sqrt{3})v+(2\sqrt{3}-7/3)w-6\sqrt{3}vw=\frac{\sqrt{3}}{3}+\frac{635}{36}.
	\end{equation}
	Hence, the matrix of the Riemannian metric is  
	\begin{align*}
		\hbar=
		\left(\begin{matrix}
			4 & 0 & 0\\
			0 & \frac{13}{16}  & -\frac{3\sqrt{3}}{16} \\
			0 &          -\frac{3\sqrt{3}}{16}  &  \frac{7}{16}
		\end{matrix}\right).
	\end{align*} 
	In the sequel, we find the wavefronts, wave rays and strategic paths for each of the cases listed above, separately.
	

	
	\begin{center}
		\textit{	Case $1$}
	\end{center}
	
	We consider the point $A$ as the origin of the coordinate system. By Theorem~\ref{main.th.1}, the wavefront at time $\tau$ is  translation of $$4\big(\frac{u}{\tau}\big)^2+\frac{13}{16}\big(\frac{v}{\tau}\big)^2+\frac{7}{16}\big(\frac{w}{\tau}\big)^2-\frac{3\sqrt{3}}{8}\big(\frac{v}{\tau}\big)\big(\frac{w}{\tau}\big)=1,$$ by  $\tau W$.

	The set of  wave rays, by item $(ii)$ of Theorem~\ref{main.th.1},  is \[\{\gamma_{_F}(t)=tV:\   |V-W|_{_h}=1\}, \] where $V=(v_1,v_2,v_3)$ and the equality $|V-W|_{_h}=1$ is equivalent to 
	\begin{equation}\label{in-ex}
		64{v_1}^2+13v_2^2+7v_3^2-(\frac{26}{3}+\sqrt{3})v_2+(2\sqrt{3}-\frac73)v_3-6\sqrt{3}v_2v_3=\frac{\sqrt{3}}{3}+\frac{635}{36}.
	\end{equation}

	In order to track the strategic path which reaches to the wavefront at time $\tau$, it suffices to find the point $q=(u,v,w)$ for which  $\sqrt{u^2+v^2+w^2}$ is maximum among all other points belonging to this wavefront. Afterwards, the strategic path is $\gamma_{_F}(t)=t\frac{q}{\tau}$. The Fig. \ref{wavefr1} depicts,  in $D$,  wavefronts at  $t=1\,,2\,,\cdots,\,10$ and the strategic path from $t=0$ to $t=10$.

	\begin{figure}[h!]
		\centering
		\includegraphics[scale=1]{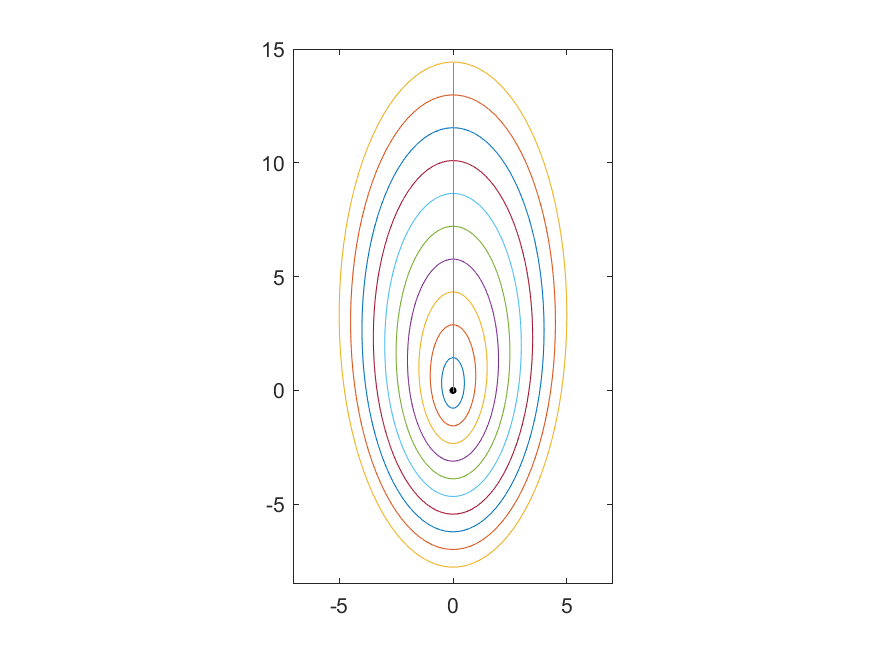}
		\caption{
		\small	some of wavefronts and strategic path in $D$ from $1$ to $10$ for  \textit{Case 1}}
		\label{wavefr1}
	\end{figure} 
	\begin{center}
		\textit{ Case $2$}
	\end{center}
	
	To find the wavefront at time $\tau$, one may use the spherical wavefronts given by Eq.~(\ref{ind.ape}) and centers at different points belonging to $C([0,\pi])$, and then apply the Huygens' principle.		
	
	The
	wave ray that emanates from some point $p = C(s_p)$  is  $\gamma_{_F}(t) = p + tV$,	
	 	  where	 	  $V=(v_1,v_2,v_3)$ is the solution of the  following  equations system  at point $p$.
	
	\begin{equation}
		\left\{\begin{array}{l}\label{syst1}
			\langle V-W, C'(s_p)\rangle=0\\
			|V-W|_{_h}=1.\\
		\end{array}\right.
	\end{equation}
	Here, in the system (\ref{syst1}), $|V-W|_{_h}=1$ is equivalent to Eq.~(\ref{in-ex}) and the first equation is equivalent to \[- v_1\big(\sin 2s_p+6\sin s_p\big)+\big(\frac{13}{16}(v_2-\frac13)-\frac{3\sqrt{3}}{16}(v_3-\frac16)\big)\frac{1}{13}\big(-\sin 2s_p+3\cos s_p
	\big)=0.\]

	In the case that all points of wavefront have the same priority, to find the strategic path that reaches  to the wavefront at time $10$   one has to find the wave ray ${\gamma}_{_F}(t)=p+tV$ for which $v_1^2+v_2^2+v_3^2$ is maximum. Fig. \ref{wavefr2} depicts some of the wavefronts from time $1$ to $10$, the strategic path (in purple color) from till time $10$, and the path (in black color) through which the fire is progressing slower.
	
	\begin{figure}[h!]
		\centering
		\includegraphics[scale=1]{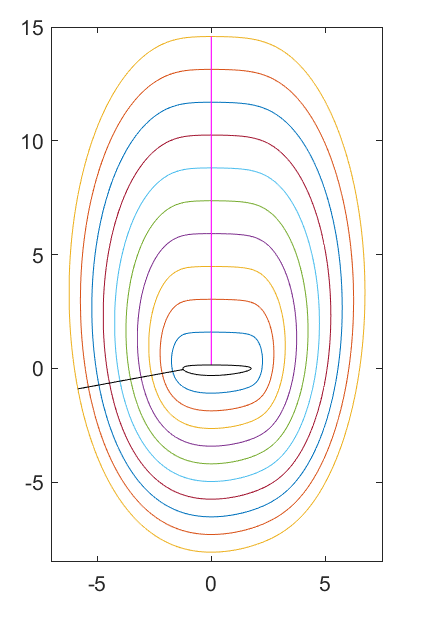}
		\caption{\small some of wavefronts from $1$ to $10$ and strategic path (purple) for~\textit{Case~2} in $D$}
		\label{wavefr2}
	\end{figure} 
	\vspace{3mm}
	\begin{center}
		\textit{Case $3$}
	\end{center}
	
	In order to find the wavefront at time $\tau$, one way is finding the spherical wavefronts with centers at points $p\in S([0,2\pi]\times [0,2])$, and then the hypersurface which is tangent to all of these spherical wavefronts. 
	
	To discover the wave ray that emanates from some point $p= S(s_{1,p},s_{2,p})$, i.e. $\gamma_{_F}(t)=p+tV$, one has to find vectors $V$ that satisfies the following system
	
	\begin{equation}\label{case2.eq}
		\left\{\begin{array}{l}
			\langle V-W, \frac{\partial S}{\partial s_1}\rangle=0,\\
			\langle V-W,\frac{\partial S}{\partial s_2}\rangle=0,\\
			|V-W|_{_h}=1,\\
		\end{array}\right.
	\end{equation} 
	provided that $V-W$ is in the direction of the propagation.
	The first equation in the system \ref{case2.eq} is equivalent to $$-\frac25v_1\sin s_1(2\cos s_1+1)+\frac{1}{128}\cos s_1(2\sin s_1+1)\big(13(v_2-\frac13)+3\sqrt{3}(v_3-\frac16)\big)=0,$$ and the second equation is equivalent to $$\frac{3\sqrt{3}}{7}(v_2-\frac13)=v_3-\frac16.$$ Once one has the wave rays, the set $\{\gamma_{_F}(\tau)=p+\tau V: \ p\in S([0,\pi]\times [0,2]) \}$ is the wavefront at time $\tau$.

	If all the points of $M$ have the same priority, the strategic path	that reaches  to the wavefront at time $\tau$  is the wave ray ${\gamma}_{_F}(t)=p+tV$ for which $v_1^2+v_2^2+v_3^2$ is maximum.
\end{example}

\begin{example}\label{exa2}
	Assume that a wildfire is spreading in some agricultural land  and the wind $W$ is blowing across it. we have  $W(p)=k(y,0,0)$, where $p=(x,y,z)\in M$ and $k$ is some small enough constant number so that the origin of  tangent space remains inside the indicatrix. Similar to the Example \ref{exa1}, we consider three different cases for the wavefront at time $0$, i.e. $ A$. In other words, $A$ is a point or some curve given by Eq.~\eqref{case1} or surface given by Eq.~(\ref{case2}). As the first step, one has to  find the indicatrix. Since the wind is some smooth vector field, by Theorems \ref{main.th.2} and \ref{main.th.2}, the indicatrix must be translation of the rotated ellipsoid, given by Eq.~(\ref{ellipp}), by the vector $W$. Keeping in mind Eq.~(\ref{ellipp}), assume that from  experimental data we are given $a=1, \ b=\frac12,\ c=2$,  $\alpha=\theta=0$, and $\beta=y$. Therefore, the indicatrix is  translation of the following ellipsoid by $W$:
	\begin{align}\label{elll}
		Q_{_h}(u,v,w)=({u\cos y+w\sin y})^2+ 4v^2+ (\frac{-u\sin y+w\cos y}{2})^2=1.
	\end{align}
	From Eq.~(\ref{elll}), the Riemannian metric is $h=P^T\mathcal{D}P$, where $\mathcal{D}=diag(1,\frac12,2)$  and $P(x,y,z)=R_y(y)$. From the fact that $\frac{\partial}{\partial x}(P_{ki}P_{kj})=0$,
	it is deduced that $W$ is a \textit{Killing} vector field. 
	It is easy to see that the flow of $W$ is as 
	$\varphi(t,p)=kt(y,0,0)+p$, where $p=(x,y,z)$. We apply Theorem~\ref{main.th.2} to deal with the propagation and obtain  wavefronts, wave rays, and strategic paths for each case.
	\vspace{3mm}
	\begin{center}
		\textit{Case $1$}
	\end{center}
	
	The wave rays emanating from some point $p$ and the unitary velocity vector  $V$  are $$\gamma_{_F}(t)=kt(y(t),0,0)+(x(t),y(t),z(t)),$$ where $\gamma_{_h}(t)=(x(t),y(t),z(t))$ is  solution of the following system  
	\begin{align}\label{eq1}
		\left\{	\begin{array}{l}
			x^{\prime\prime}+y^{\prime}z^{\prime}=0,\\
			y^{\prime\prime}=0,\\
			z^{\prime\prime}-x^{\prime}y^{\prime}=0,\\
		\end{array}	\right.\end{align}
	with the initial conditions $\gamma_{_h}(0)=p=(x,y,z)$, $\gamma_{_h}'(0)=V-W$ and $|V-W|_{_h}=1$. After some calculations, if $v_2\neq 0$, we have
	
	\begin{align}\label{eq2}
		\left\{	\begin{array}{l}
			x(t)=x-\frac{v_3}{v_2}+t\frac{v_3}{v_2}\cos v_2 + t\frac{v_1-ky}{v_2}\sin v_2,\\
			y(t)=tv_2+y,\\
			z(t)=z-\frac{v_1-ky}{v_2}-t\frac{v_1-ky}{v_2}\cos v_2+t\frac{v_3}{v_2}\sin v_2.\\
		\end{array}	\right.\end{align} and if $v_2=0$, 
	
	\begin{align}\label{eq3}
		\left\{	\begin{array}{l}
			x(t)=(v_1-ky)t+x,\\
			y(t)=y,\\
			z(t)=v_3t+z,\\
		\end{array}	\right.\end{align}
	are component of $\gamma_{_h}(t)$, provided that $|V-W|_{_h}=1$. One can easily verify that $|V-W|_{_h}=1$ is equivalent to 
	\begin{equation}
		4(v_1-ky)^2+\frac{13}{16}v_2^2-\frac{3\sqrt{3}}{8}v_2v_3+\frac{7}{16}v_3^2=1.
	\end{equation}
	
	Once one has the curves $\gamma_{_h}(t)=(x(t),y(t),z(t))$,  which are satisfied in the systems \ref{eq2} or \ref{eq3}, the wavefront at time $\tau$ is the set
	$$\{k\tau(y(\tau),0,0)+(x(\tau),y(\tau),z(\tau))\}.$$

	The strategic path for the situation that all  points belonging to the space have the same priority is the wave ray $\gamma_{_F}(t)=kt(y,0,0)+\gamma_{_h}(t)$ provided that $|k\tau(y,0,0)+\gamma_{_h}(\tau)-p|$ is maximum among all the curves $\gamma_{_h}(t)$. If one wants to find the wave ray that reaches to some point $q$  belonging to the wavefront at time $\tau$, it is   $\gamma_{_F}(t)=kt(y,0,0)+\gamma_{_h}(t)$ such that $\gamma_{_h}(\tau)=q-k\tau(y,0,0)$.

	\vspace{3mm}
	\begin{center}
		\textit{Case $2$}
	\end{center}
	
	For Case $2$, the wave ray which emanates from some point $p= C(s_p)$ is  $\gamma_{_F}(t)=kt(y(t),0,0)+(x(t),y(t),z(t))$ such that $\gamma_{_h}(t)=(x(t),y(t),z(t))$ is a solution of the systems  \ref{eq2} or \ref{eq3} and it also satisfies $|\gamma'_{_h}(0)|=1$, $\gamma_{_h}(0)=p$, and $\gamma_{_h}'(0)\underset{h}{\perp} C'(s)$  in the direction of the propagation.  
	
	The wavefront at time $\tau$ is the set
	
	$$\{k\tau(y(\tau),0,0)+(x(\tau),y(\tau),z(\tau)) \}.$$ 
	
	The calculations related to  strategic paths, and  wave rays and wavefront in the Case $3$ are similar to the previous case and the Case $3$ of  Example \ref{exa1}.

\end{example}

\section{Conclusion and Final Remarks}

In this work, we presented a model for wildfire  propagation  in some smooth manifold  under the presence of  wind. The results and proofs can be generalized to any dimension $n$, however to be closer to reality, we focused on the dimension $3$.  Three different cases for the wind were considered; a constant, \textit{Killing}, and smooth vector field. These three cases were separately studied  through Theorems \ref{main.th.1}, \ref{main.th.2}, and \ref{mainthm3}. This way, equations of  wavefronts and wave rays were provided. To summarize the results of these theorems, for a wildfire spreading in some space, we provided two ways to find the model of propagation as follows.
\begin{itemize}

	
	\item [(i)] \textit{Using the Huygens' principle}. It means, given the wavefront at some time $t_0\geq 0$, we find the spherical wavefronts of radius $r>0$ and center of each at some point belonging to this wavefront. Then we find the envelope of these spherical wavefronts; that is the wavefront at time $t_0+r$. This method is more applicable when we want to use some computer programming and software for modeling the propagation. 
	
	\item [(ii)] \textit{Using geodesics}. The unitary  Randers geodesics that start  orthogonally from some wavefront are  wave rays  and all of them reach to the next wavefront at the same time. It means given some wavefront $A$, after time $\tau$, the wavefront will be the set  $$\{{\gamma_{_F}}(\tau)\ \},$$ where $\gamma_{_F}(t)$ are unitary Randers geodesics that emanates from $A$ orthogonally.

\end{itemize}

One can find Randers geodesics without even calculating the metric. In Theorems \ref{main.th.1} and \ref{main.th.2}, the formulas of Randers geodesics, for the cases of wind being a constant or \textit{Killing} vector field, are provided. For the case of  wind being a smooth vector field, by Theorem~\ref{mainthm3} and by formulas and relations provided in Section 2.3 of \cite{cheng2012finsler}, one obtains the system of Randers geodesic equations from the system of  Riemannian geodesic equations. Therefore, given a propagation, it suffices to find the Riemannian metric, and then the system of geodesic equations  associated with it, and finally the wavefronts. 

By the way, concepts of strategic paths and points were introduced and their equations were determined  for different situations. These paths and points are those which are of great importance in wildfire management strategies; by providing the places where the firefighters and equipment should be located. Two examples, in which some wildfire is spreading in some agricultural land, were given to illustrate the results.

In this work, it was assumed that the wind $W$ is  time-independent. However, one can apply our results for the case that the wind is some time-dependent vector field $W(t,p),$  $t\in[0,s]$, assuming that $\{t_1,....,t_n\}$ is a partition of $[0,s]$ such that at each interval  $[t_i,t_{i+1}]\subset [0,s]$, the wind is the time-independent vector field $W^i(p)$. Therefore, at each interval $[t_i,t_{i+1}]$, we are faced with a new propagation and have to apply one of the Theorems \ref{main.th.1}, \ref{main.th.2} or \ref{mainthm3} to modeling it. Generally, there is no relations between the  metric  related to one interval with that of another one.

By the way, in the case that we also need the model of wildfire propagation  in dimension $2$, that is on the land, it suffices to consider the intersection of  wavefronts/propagation  with the land.

For further works in the future, one may work on the cases that the wildfire creates singularities or cut loci. Also , it is interesting to see how the wildfire spreads when the fuel has not been  distributed smoothly through space.  Moreover, studying the propagation of  wildfire on a mountain or some slope is welcome as it is related to certain  real situations. In this regard, one problem can be finding the relations between propagation of wildfires on two domains under the same conditions but different slopes. By the way, as another  problem, one may study the model associated with a strong wind, that is when the origin of tangent space does not locate inside the indicatrix.

\renewcommand{\baselinestretch}{.5}
\bibliography{bibliojournalversion}
\bibliographystyle{acm}

\end{document}